\documentclass[dvips,preprint]{imsart}
\usepackage{amsthm,amsmath}

\usepackage[T1]{fontenc}
\usepackage[dvips]{graphics}
\usepackage{pifont}
\usepackage{amsfonts}
\usepackage{latexsym}
\usepackage{amsmath}
\usepackage{amssymb}
\usepackage{amscd}
\usepackage{color}
\usepackage{verbatim}
\usepackage{dsfont,bbold}

\startlocaldefs

\renewcommand{\P}{{\rm P}}
\newcommand{\E}{{\rm E}}
\renewcommand{\sp}{\mathrm{sp}}
\newcommand{\tr}{^{\mbox{\tiny T}}}

\newcommand{\diag}{\mathrm{diag}}
\newcommand{\re}{\mathrm{Re}}
\newcommand{\ud}{\, \mathrm{d}}  
\def\C{\mathds C}

\def\FF{{\cal F}}
\def\GG{{\cal G}}
\def\HH{{\cal H}}

\renewcommand{\u}{{\mbox{\tiny $+$}}}
\renewcommand{\d}{{\mbox{\tiny $-$}}}

\newcommand{\vect}[1]{\boldsymbol #1}

\newcommand{\valpha}{\vect \alpha}

\newcommand{\vgamma}{\vect \gamma}

\newcommand{\vone}{\vect 1}
\newcommand{\vzero}{\vect 0}

\newcommand{\ve}{\vect e}

\newcommand{\vp}{\vect p}

\newcommand{\Psist}{\Psi^*}
\newcommand{\epsi}{\varepsilon}

\newcommand{\vligne}[1]{\begin{bmatrix} #1 \end{bmatrix}}

\newtheorem{defn}{Definition}[section]
\newtheorem{lem}[defn]{Lemma}
\newtheorem{prop}[defn]{Proposition}
\newtheorem{theo}[defn]{Theorem}
\newtheorem{thm}[defn]{Theorem}
\newtheorem{cor}[defn]{Corollary}
\newtheorem{rem}[defn]{Remark}
\newtheorem{ass}[defn]{Assumption}

\newcommand{\debproof}{\begin{proof}}
\newcommand{\finproof}{\end{proof}}

\newcommand{\bs}{\boldsymbol}
\endlocaldefs

\begin{document}

\begin{frontmatter}

\title{The morphing of fluid queues 
into Markov-modulated Brownian motion}
\runtitle{The morphing of fluid queues into MMBM}


\author{\fnms{Guy} \snm{Latouche}\ead[label=e1]{guy.latouche@ulb.ac.be}
}
\thankstext{t1}{The authors had interesting discussions with V. Ramaswami over the
questions examined in this paper. They thank three anonymous referees for their careful analysis and insightful suggestions, which helped streamline some of the proofs and improve the overall presentation of the paper. 
They also thank the Minist\`ere de la
Communaut\'e fran\c{c}aise de Belgique for funding this research
through the ARC grant AUWB-08/13--ULB~5, and acknowledge the
financial support of the Australian Research Council through the
Discovery Grant DP110101663.} 
\address{Universit\'{e} libre de Bruxelles \\ 
D\'{e}partement d'Informatique \\ \printead{e1}}
\author{\fnms{Giang} \snm{Nguyen}\corref{}\ead[label=e2]{giang.nguyen@adelaide.edu.au}\thanksref{t1}}
\address{The University of Adelaide \\ School of Mathematical Sciences \\ \printead{e2}}

\runauthor{G. Latouche and G. T. Nguyen}

\begin{abstract}
Ramaswami showed recently that standard Brownian motion arises as the
limit of a family of Markov-modulated linear fluid processes. We
pursue this analysis with a fluid approximation for Markov-modulated
Brownian motion. We follow a Markov-renewal approach and  we prove
that the stationary distribution of a Markov-modulated Brownian motion
reflected at zero is the limit from the well-analyzed stationary
distribution of approximating linear fluid processes. Thus, we provide
a new approach for obtaining the stationary distribution of a
reflected MMBM without time-reversal or solving partial differential
equations.  Our results open the way to the analysis of more complex Markov-modulated processes.

Key matrices in the limiting stationary distribution are shown to be solutions of a matrix-quadratic equation, and we describe how this equation can be efficiently solved. 
\end{abstract}



\end{frontmatter}



\section{Introduction}
	
Our purpose is to construct and analyse a family of fluid queues
converging to Markov-modulated Brownian motion (MMBM) with the
intention of adapting, to the analysis of MMBM, tools and methods
which have been developed in the context of fluid queues.

Fluid queues are two-dimensional processes $\{X(t), \varphi(t): t \geq 0\}$, where $\{\varphi(t): t \geq 0\}$ is a continous-time Markov chain on a finite state space $\mathcal{S}$, 
\begin{align*}
X(t) = X(0) + \int_0^t c_{\varphi(t)} \ud t,  
\end{align*} 
and $c_i$ for $i \in \mathcal{M}$ are arbitrary real numbers. These
are also known as \emph{Markov-modulated linear fluid
  processes}, with $X$ referred to as the \emph{fluid level} and
$\varphi$ as the \emph{phase}: during intervals of time where the
phase $\varphi$ remains in some state $i \in \mathcal{M}$, the
fluid level varies linearly at the rate $c_i$. The associated process
reflected at zero is denoted by $\{\widehat{X}(t), \varphi(t): t \geq
0\}$, where
\begin{align*}
\widehat{X}(t) = X(t) - \inf_{0 \leq v \leq t} X(v),  
\end{align*} 
assuming that $\widehat{X}(0) = 0$. Also referred to as {first-order
  fluid processes}, these stochastic models are useful when the
relevant rates of change can be well-described by their first moments.

{Markov-modulated Brownian motion} (MMBM), or the family of
{second-order fluid processes}, takes into account first and second
moments of the rates of change. In particular, the fluid level $Y(t)$
of a Markov-modulated Brownian motion $\{Y(t), \kappa(t): t \geq 0\}$
is a Brownian motion with drift $\mu_i$ and variance $\sigma^2_i$
during time intervals where $\kappa(t) = i$; one sometimes writes that
\[
Y(t) = Y(0) + \int_0^t c_{\kappa(t)} \ud t + \int_0^t \sigma_{\kappa(t)} \ud W(t)
\]
where $\{W(t)\}$ is a standard Brownian motion.

Three papers appeared in close succession on the stationary
distribution of MMBM reflected at zero: Rogers~\cite{roger94}, Asmussen~\cite{asmussen95}, and Karandikar and
Kulkarni~\cite{kk95}.  The focus in the third paper is on solving
partial differential equations, and it is not of further concern to us
in the present paper. In \cite{asmussen95,roger94}, on the other
hand, the authors obtain the stationary distribution in a form which
is suitable for calculations with linear algebraic procedures.  These
results crucially depend on the technique of reversing time, a method
already used in Loynes~\cite{loyne62b} whereby the stationary
distribution of the process is obtained from the distribution of the
maximum of a random walk with negative drift.
More recent work for obtaining the stationary distribution of 
Markov-modulated L\'{e}vy processes with reflecting boundaries, as in
Asmussen and Kella~\cite{ao00}, Ivanovs~\cite{ivanovs-japs10}, D'Auria
\emph{et al.}~\cite{dikm12}, and D'Auria and Kella~\cite{dk12}, also
uses the reverse-time approach.

For fluid processes without a Brownian component, another line of
investigation was open in Ramaswami~\cite{ram99}, based on
renewal-type arguments similar to the ones used in the analysis of
quasi-birth-and-death processes (Neuts~\cite[Chapter 3]{neuts81},
Latouche and Ramaswami~\cite[Chapter~6]{lr99}).  This eventually led,
in addition to interesting algorithmic procedures, to theoretical
developments for fluid processes in finite time, and for systems with more
complex interactions between phase and level.    There, the flow of
time is not reversed and this creates a significant difference in the
methods of analysis.

Our intention is to establish a link between the results for fluid
queues and those for MMBM.
In Section~\ref{sec:fluid-to-mmbm}, we extend the argument from
Ramaswami~\cite{ram11} and define a parameterised family of linear
fluid processes that converge weakly, as the parameter tends to
infinity, to a Markov-modulated Brownian motion.   Ahn and Ramaswami~\cite{ra13} are independently using matrix-analytic methods to analyze Markov-modulated Brownian motions, with different approaches to ours. 
We determine in Sections~\ref{sec:statdis} the limiting structure of
key matrices and quadratic matrix equations, and we establish the
connection between the stationary distribution so obtained, and the
one which follows from the time-reversed approach.  We present in
Section~\ref{sec:comproc} a computational procedure for solving the
quadratic equation efficiently.

\section{Markov-modulated Brownian motion}   
	\label{sec:fluid-to-mmbm}

We show here that a family of linear fluid processes converges weakly
to a Markov-modulated Brownian motion $\{Y(t), \kappa(t): t \geq 0\}$,
where the phase process $\kappa$ is a Markov chain with state space
$\mathcal{M} = \{1, \ldots, m\}$, and $Y$ is a Brownian motion with
drift $\mu_i$ and variance $\sigma^2_i$ whenever $\kappa(t) = i \in
\mathcal{M}$. We denote by $D$ the drift matrix $\diag (\mu_1, \ldots,
\mu_m)$, by $V$ the variance matrix $\diag(\sigma^2_1, \ldots,
\sigma^2_m)$, and by $Q$ the generator of $\kappa$, and we assume that
$Q$ is irreducible. 

\begin{ass} At time 0, the level $Y(0)$ is equal to 0,
  and the initial phase $\kappa(0)$ has the stationary distribution
  $\valpha$ of $Q$ ($\valpha Q = \vzero$ and $\valpha \vone = 1$).
\end{ass} 
Formally, for $t \geq 0$ the process $Y(t)$ can be defined recursively as 
\begin{align}
	\label{eqn:recY} 
Y(t) = Y(T) + \sum_{i \in \mathcal{M}} \mathds{1}_{\{\kappa(T) = i\}}\{Y_i(t) - Y_i(T)\}, 
\end{align} 
where $Y(0) = 0$, the random variable $T$ is the last jump epoch of $\kappa$ before~$t$ ($T = 0$ if there has yet to be a jump), the process $Y_i$ is a Brownian motion with mean $\mu_i$ and variance $\sigma_i^2$, and $\mathds{1}_{\{\cdot\}}$ denotes the indicator function. The processes $Y_i$ for al $i \in \mathcal{M}$ and $\varphi$ are assumed to be mutually independent.  

We construct the family of fluid processes $\{L_{\lambda}(t),\beta_\lambda(t), \varphi_\lambda(t): t \geq 0\}$ as
follows: the phase process is a two-dimensional Markov chain $(\beta_\lambda(t), \varphi_\lambda(t))$ with state space $\mathcal{S}
= \{(k,i): k \in \{1,2\} \mbox{ and } i \in \mathcal{M}\}$ and generator
\begin{align*}
T_{\lambda} = \left[\begin{array}{cc} Q - \lambda I & \lambda I \\
\lambda I & Q - \lambda I
\end{array}\right], 
\end{align*} 
where the entries of~$T_\lambda$ follow the lexicographic ordering of
$\{1,2\} \times \mathcal{M}$, and $I$ denotes an appropriately-sized
identity matrix. Whenever ambiguity might arise, we write $I_n$ to
denote the $n \times n$ identity matrix. The fluid rate matrix
$C_{\lambda}$ is given by
\begin{align*}
C_{\lambda} = \left[\begin{array}{cc} 
D + \sqrt{\lambda} \Theta & \\
& D - \sqrt{\lambda} \Theta 
\end{array}\right],  \quad \mbox{where $\Theta = \sqrt{V}$.}
\end{align*} 
%
%
%
\begin{ass} \label{ass:fluid}  At time $0$, the level
$L_{\lambda}(0)$ is equal to $0$, and the initial phases $\beta_{\lambda}(0)$
and $\varphi_{\lambda}(0)$ have their respective stationary
distributions $\boldsymbol{\gamma} = (1/2, 1/2)$ and
$\boldsymbol{\alpha}$, their joint distribution is $\vp = \vgamma
\otimes  \boldsymbol{\alpha}$. 
\end{ass} 

Intuitively speaking, we duplicate the state space $\mathcal{M}$ in the Markov-modulated Brownian motion $\{Y(t), \kappa(t)\}$ and the auxiliary process $\beta_\lambda(t)$ keeps track of which copy is in use.
Note that for~$\lambda$ sufficiently large, the phases in the copy
with $\beta_{\lambda}(t) = 1$ have all positive rates while the phases
in the other copy have all negative rates.   With this construction, we  show that the conditional moment generating function of $\{L_{\lambda}(t), \varphi_\lambda(t)\}$ converges to that of $\{Y(t),\kappa(t)\}$. 

\begin{rem} \em
	\label{rem:reurs}
Based on the recursive definition~\eqref{eqn:recY} of $Y$, an
alternative interpretation for our fluid-based approximation is that
for each phase $i \in \mathcal{M}$ we approximate the process $Y_i$ by
a two-phase fluid process $\{L_{\lambda}^{i}(t),
\beta_{\lambda}^{i}(t)\}$ 
the phase $\beta_{\lambda}^{i}$ is a Markov chain with state space $\mathcal{S}^{i} = \{1,2\}$ and generator
\begin{align*}
T_{\lambda}^{i} = \left[\begin{array}{rr} -\lambda & \lambda \\ 
\lambda & -\lambda \end{array} \right],  
\end{align*} 
and the fluid rate matrix $C_{\lambda}^{i}$ is given by 
\begin{align*}
C_{\lambda}^{i}  = \left[\begin{array}{cc} \mu_i + \sqrt{\lambda}\sigma_i &  \\ 
& \mu_i - \sqrt{\lambda}\sigma_i \end{array}\right],
\end{align*} 
the approximating processes being independent.
\end{rem}
%
%
%

Denote by $\Delta_Y(s)$ the $m \times m$ Laplace matrix exponent of $\{Y(t), \kappa(t)\}$, by $\tilde{\Delta}_{\lambda}(s)$ the $2m \times 2m$ Laplace matrix exponent of $\{L_{\lambda}(t), \beta_{\lambda}(t), \varphi_{\lambda}(t)\}$, and by $\Delta_{\lambda}(s)$ the $m \times m$ Laplace matrix exponent of $\{L_{\lambda}(t), \varphi_{\lambda}(t)\}$. These matrices are such that 
\begin{align}
[e^{\Delta_Y(s)t}]_{ij} & = \E[e^{sY(t)}\mathds{1}_{\{\kappa(t) = j\}}|Y(0) = 0, \kappa(0) = i], \nonumber \\
[e^{\tilde{\Delta}_{\lambda}(s)t}]_{(k,i)(k',j)} & = \E[e^{sL_{\lambda}(t)}\mathds{1}_{\{\beta_{\lambda}(t) = k', \varphi_{\lambda}(t) = j\}}| L_{\lambda}(0) = 0, \beta_{\lambda}(0) = k, \varphi_{\lambda}(0) = i], \nonumber \\
\intertext{and} 
[e^{\Delta_{\lambda}(s)t}]_{ij} & = \E[e^{sL_{\lambda}(t)}\mathds{1}_{\{\varphi_{\lambda}(t) = j\}}| L_{\lambda}(0)  = 0, \varphi_{\lambda}(0) = i]. \label{eqn:Lme3} \end{align} 
By Asmussen and Kella~\cite{ao00}, the Laplace matrix exponent of a Markov-modulated L\'{e}vy process $\{Z(t), \xi(t)\}$ with jumps is given by 
\begin{align*}
\Delta_Z(s) = \diag(\phi_1(s), \ldots, \phi_p(s)) + Q \circ R(s), 
\end{align*} 
where $\phi_i(s)$ is the Laplace exponent of an unmodulated L\'{e}vy process with parameters defined for phase $i \in \{1, \ldots, p\}$, $Q$ is the phase-transition matrix of $\xi(t)$, $R$ is the matrix with components $[R(s)]_{ij} = \E[e^{sW_{ij}}]$, which are the Laplace transforms of the jumps $W_{ij}$ for $Z(t)$ when $\xi(t)$ moves from $i$ to $j$, and $\circ$ indicates the Hadamard product. 

As the Laplace exponent of an unmodulated Brownian motion with drift $\mu$ and variance $\sigma^2$ is given by $\mu s + \sigma^2 s^2/2$, one can verify that
\begin{align}
\nonumber
 \Delta_Y(s) &= sD + (s^2/2) V + Q,  \\
\nonumber
\tilde{\Delta}_{\lambda}(s) &= sC_{\lambda} + T_{\lambda}  \\
\nonumber
 & = I \otimes M + s \sqrt{\lambda} J \otimes \Theta + \lambda G \otimes I,
\intertext{where $M=sD+Q$, and} 
  \label{eqn:eDelta3}
e^{\Delta_{\lambda}(s)t} &= (\bs{\gamma} \otimes I)e^{\tilde{\Delta}_{\lambda}(s)t}(\bs{1} \otimes I),
\end{align} 
where $\bs{1}$ is an appropriately-sized column vector of ones.

The next lemma gives a technical property of the matrix
exponential, it is used in the proof of Theorem~\ref{theo:mgfMMBM}.

\begin{lem}
   \label{t:corner}
Let $S$ be the block-partitioned matrix
\[
S = \vligne{S_{11} & S_{12} \\ S_{21}  & S_{22}}
\]
where $S_{11}$ and $S_{22}$ are matrices of order $m_1$ and $m_2$, respectively.  Denote by $H(t)$ the north-west quadrant of order $m_1$ of $e^{S t}$:
\[
H(t) = \vligne{I_{m_1 \times m_1} & 0}  e^{S t} \vligne{I_{m_1 \times m_1} \\ 0}.
\]

The matrix $H(t)$ is the solution of 
\begin{equation}
   \label{e:cornera}
H(t) = e^{S_{11}t} + \int_0^t  \int_v^t e^{S_{11}(t-u)} S_{12}    e^{S_{22}(u-v)} S_{21} H(v) \, \ud u \, \ud v.
\end{equation}
\end{lem}

\begin{proof}
We decompose $S$ as the sum $S=S_A + S_E$, with
\[
S_A = \vligne{S_{11} & S_{12} \\ 0 & S_{22}} \qquad \mbox{and} \qquad 
S_E = \vligne{0 & 0 \\ S_{21} & 0} = \vligne{0 \\ I}  S_{21} \vligne{I & 0} .
\]
By Higham~\cite[Equations (10:13) and (10:40)]{higham08}, we obtain
\[
e^{S t} = e^{S_A t}  + \int_0^t  e^{S_A(t-v)} S_E e^{S (v)} \, \ud v.
\]
and
\[
e^{S_A t} = \vligne{
e^{S_{11}t}  & \int_0^t  e^{S_{11}(t-u)} S_{12} e^{S_{22}u} \, \ud u\\
0 & e^{S_{22}t}
}
 \]
Thus,
\begin{eqnarray*}
H(t) & = & 
  \vligne{I & 0}  e^{S_A t} \vligne{I \\ 0} 
  + \int_0^t \vligne{I & 0}  e^{S_A(t-v)} S_E e^{S (v)} \vligne{I \\ 0} \, \ud v
\\
 &=& e^{S_{11}t} + \int_0^t \vligne{I & 0}  e^{S_A(t-v)} \vligne{0 \\ I}  S_{21} \vligne{I & 0} e^{S (v)} \vligne{I \\ 0} \, \ud v
\\
  &=& e^{S_{11}t} + \int_0^t \int_0^{t-v} e^{S_{11}(t-v-u)} S_{12} e^{S_{22}u} S_{21} H(v)  \, \ud u  \, \ud v
\end{eqnarray*}
which proves (\ref{e:cornera}).
\end{proof}

\begin{theo} \label{theo:mgfMMBM} 
The conditional moment generating function of $\{L_{\lambda}(t),\varphi_{\lambda}(t)\}$ converges to that of $\{Y(t),\kappa(t)\}$, that is, 
\begin{align} \label{e:limit}
\lim_{\lambda \rightarrow \infty} (\bs{\gamma} \otimes I)e^{\tilde{\Delta}_{\lambda}(s)t}(\bs{1} \otimes I) & = e^{\Delta_Y(s)t}. 
\end{align} 
\end{theo} 

\begin{proof}
We proceed in three steps. 
First, we observe that
\begin{align}
   \label{e:fk}
\tilde{\Delta}_\lambda^k(s) (\vone \otimes I) = \vone \otimes A_k  + \ve \otimes B_k, \qquad \mbox{for $k \geq 0$,}
\end{align}
where $\bs{e} = (1,-1)^{\mbox{\tiny T}}$ and 
\begin{align}
   \label{e:akbk}
\vligne{A_k \\ B_k} = 
  \Upsilon^k
  \vligne{I \\ 0},
\end{align}
with
\begin{align} 
\nonumber
\Upsilon = \vligne{M  & s \sqrt{\lambda} \Theta \\ s \sqrt{\lambda} \Theta & M-2 \lambda I}.
\end{align} 
The proof of \eqref{e:fk} is by induction: that equation trivially holds for $k=0$, with
$A_0 = I$ and $B_0 =0$ and, if it also holds for a given value of $k$, then  we
 easily verify that 
$\tilde{\Delta}_\lambda^{k+1}(s) (\vone \otimes I) = \vone \otimes A_{k+1}  + \ve \otimes B_{k+1}$  with 
\begin{align*} 
A_{k+1}  = M A_k + s \sqrt{\lambda} \Theta B_k, \qquad
B_{k+1}  =  s \sqrt{\lambda} \Theta A_k + (M-2\lambda I) B_k,
\end{align*} 
or
\begin{align*} 
\vligne{A_{k+1} \\ B_{k+1}} =
\vligne{ M & s \sqrt{\lambda} \Theta \\
   s \sqrt{\lambda} \Theta & M-2\lambda I}
\vligne{A_{k} \\ B_{k}}.
\end{align*}
Equation (\ref{e:akbk}) readily follows.

To simplify the notation, we define $H_\lambda(t) =
e^{\Delta_\lambda(s) t}$ for the remainder of the proof.  By
(\ref{eqn:eDelta3}), we have
\begin{align}
\nonumber
 { H}_\lambda(t)
& = \frac{1}{2} (\vone\tr \otimes I) \sum_{k = 0}^{\infty} \frac{t^k}{k!} \tilde{\Delta}_{\lambda}^{k}(s) (\vone \otimes I) \nonumber \\
& = 
  \sum_{k \geq 0}  \frac{t^k}{k!}  A_k \nonumber \qquad\qquad  \mbox{from (\ref{e:fk})} \\
\nonumber
& = \vligne{I & 0} \exp{\Upsilon t} \vligne{I \\ 0}.
\end{align}
By Lemma~\ref{t:corner},  $H_\lambda(t)$ is a solution of
\begin{align}
   \label{e:cornerc}
H_\lambda(t) = e^{Mt} + s^2 \lambda \int_0^t \left\{ \int_v^t e^{M(t-u)} \Theta  e^{(M-2\lambda I)(u-v)} \ud u \right\} \Theta H_\lambda(v)  \ud v.
\end{align}
Integrating by parts the inner integral, we find 
\begin{align*}
\lambda   \int_v^t & e^{M(t-u)} \Theta  e^{(M-2\lambda I)(u-v)} \ud u  \\
 =  &\left[ \lambda e^{M(t-u)} \Theta  e^{(M-2\lambda I)(u-v)} (M-2\lambda I)^{-1} \right]_v^t \\
&  + \lambda  \int_v^t M e^{M(t-u)} \Theta  e^{(M-2\lambda I)(u-v)} (M-2\lambda I)^{-1} \ud u 
\\
 = & \; \Theta  e^{(M-2\lambda I)(t-v)} (1/\lambda M-2 I)^{-1} - e^{M(t-v)} \Theta (1/\lambda M-2 I)^{-1} \\
&  + \int_v^t M e^{M(t-u)} \Theta  e^{(M-2\lambda I)(u-v)} (1/\lambda M-2 I)^{-1} \ud u,
\end{align*}
which converges to $1/2  e^{M(t-v)} \Theta $ as $\lambda \rightarrow \infty$. From (\ref{e:cornerc}), we conclude that 
\begin{align*} 
H_\infty(t) = e^{Mt} + (s^2 /2) \int_0^t  e^{M(t-v)} V H_\infty(v)\ud v,
\end{align*}
and therefore
\begin{align*}
\frac{\ud}{\ud t} H_\infty(t) & = (M + (s^2/2) V) H_\infty(t) = \Delta_{Y}(s) H_\infty(t).
\end{align*}
This completes the proof. 
\end{proof}

The Laplace matrix exponent uniquely characterizes the finite-dimensional distributions of the process and therefore Theorem~\ref{theo:mgfMMBM} implies the following result.
		
\begin{cor} \label{cor:findimMMBM} 
The finite-dimensional distributions of $\{L_{\lambda}(t), \varphi_\lambda(t)\}$ converge to the finite-dimensional distributions of the Markov-modulated Brownian motion $\{Y(t), \kappa(t)\}$. 
\qed
\end{cor}

\begin{theo} \label{theo:tightness}
The family $\{L_{\lambda}(t), \varphi_{\lambda}(t): t \geq 0\}$ is tight. 
\end{theo}	

\begin{proof} In this proof, we use alternative interpretation of our fluid-based approximation as outlined in Remark~\ref{rem:reurs}. 

The proof of~\cite[Theorem~5]{ram11} shows that for each $i \in \mathcal{M}$ the family of marginal processes $\{L^{i}_{\lambda}\}$ is tight on any compact interval of $[0,\infty)$. By~\cite[Corollary~7]{whitt70}, we can extend this result to show that the family of marginal processes $\{L^{i}_{\lambda}\}$ is tight on $[0,\infty)$. 

As the family of $\{L_{\lambda}^i\}$ converges weakly to the Brownian motion $\{Y^i\}$ with $\mu_i$ and $\sigma_i$, so do the processes $\{\sup_{s < t} |L_{\lambda}^i(t)|\}$ to the corresponding supremum process $\{\sup_{s < t} |Y^i(t)|\}$, by the continuous mapping theorem. Thus, the family $\{\sup_{s < t} |Y^i(t)|\}$ is tight on $[0,\infty)$. 

Next, we define $\widetilde{L}_{\lambda}(t)$ for $t \geq 0$ as 
\begin{align}
	\widetilde{L}_{\lambda}(t) = \sum_{i = 1}^{m}\sup_{s \leq t} |L_{\lambda}^i(t)|.  
\end{align} 
As $L_\lambda$ is stochastically dominated by $\widetilde  L_{\lambda}$, the tightness property of the family of $\{L_{\lambda}\}$ follows from the tightness property of the family $\{\sup\limits_{s < t} L_{\lambda}^{i}(t)\}$ for all $i \in \mathcal{S}$.  \end{proof} 

The next theorem follows from Corollary~\ref{cor:findimMMBM} and Theorem~\ref{theo:tightness}.
\begin{theo}\label{theo:wc} 
The processes $\{L_{\lambda}(t),\varphi_{\lambda}(t): t \geq 0\}$
converge weakly to the Markov-modulated Brownian motion
$\{Y(t),\kappa(t): t \geq 0\}$. 
\qed
\end{theo}

\section{Stationary distribution}
   	\label{sec:statdis} 
       
We consider again the Markov-modulated Brownian motion $\{Y(t),
\kappa(t): t \geq 0\}$ described in Section~\ref{sec:fluid-to-mmbm},
but with a reflection at level zero. The reflected process is denoted
as $\{\widehat{Y}(t), \kappa(t): t \geq 0\}$, with
\begin{align*}
\widehat{Y}(t) = Y(t) - \inf_{0 \leq v \leq t} Y(v).
\end{align*}
Furthermore, we define the
reflected fluid process $\{\widehat{L}_\lambda(t), \beta_{\lambda}(t),
\varphi_{\lambda}(t): t \geq 0\}$, where 
$$\widehat{L}_{\lambda}(t) =
L_{\lambda}(t) - \inf_{0 \leq v \leq t}L_{\lambda}(v).$$

For notational convenience,  we define
$\varepsilon = 1/\sqrt{\lambda}$.    With this, the reflected fluid
processes is written as $\{\widehat{L}_\epsi(t), \beta_{\epsi}(t),
\varphi_{\epsi}(t): t \geq 0\}$ and our purpose is to show that the
stationary distribution of $\{\widehat{Y}(t), \kappa(t)\}$ is the
limit, as $\epsi \rightarrow 0$, of the stationary distribution of the
reflected fluid process $\{\widehat{L}_\epsi(t),
\varphi_{\epsi}(t)\}$.  We emphasize that the
processes $\{\widehat{L}_\lambda(t),
\varphi_{\lambda }(t)\}$ and $\{\widehat{L}_\epsi(t),
\varphi_{\epsi}(t)\}$ are the same. The change in subscripts only
reflects the notational change in our perturbation parameter.

%

\begin{ass} 
   \label{a:drift}
%
The mean drift  $\valpha D \vone$ is negative, so
   that all reflected processes are positive recurrent.
\end{ass} 

\begin{ass} 
	\label{ass:possigma}
 The variance $\sigma^2_i$ is positive for all $i \in
  \mathcal{M}$. This assumption ensures the existence of $\Theta^{-1}$, which we need later on. \end{ass}

The following result is a direct corollary of Theorem~\ref{theo:wc}, by the Skorokhod mapping theorem. 

\begin{cor} The processes
  $\{\widehat{L}_{\epsi}(t),\varphi_{\epsi}(t): t \geq 0\}$ weakly
  converge as $\epsi \rightarrow 0$  to the reflected
  Markov-modulated Brownian motion $\{\widehat{Y}(t),\kappa(t): t \geq
  0\}$.  
\qed
\end{cor} 

We denote the stationary distribution vector of $\{\widehat{L}_{\varepsilon}(t), \beta_{\varepsilon}(t), \varphi_{\varepsilon}(t)\}$ by $\bs{F}_{\varepsilon}(x)$
and the associated stationary density vector by $\bs{\pi}_{\varepsilon}(x)$, with components 
\[
[F_{\varepsilon}(x)]_{ki}  = \lim_{t \rightarrow \infty} \P[\widehat{L}_\epsi(t) \leq x, \beta_{\varepsilon}(t) = k, \varphi_{\varepsilon}(t) = i],
\]
{and} 
$ [{\pi}_{\varepsilon}(x)]_{ki}  = \ud/\hspace*{-0.05cm}\ud x
[F_{\varepsilon}(x)]_{ki}$, for $k \in \{1,2\}$ and $i \in
\mathcal{M}$, and we partition the generator and the fluid rate matrices as
\begin{align*} 
T_\epsi = \vligne{T_{\epsi}^{++} &  T_{\epsi}^{+-}  \\ T_{\epsi}^{-+}  & T_{\epsi}^{--} }
\qquad \mbox{and} \qquad
C_\epsi = \vligne{C_{\epsi}^{+} & 0 \\ 0 & C_{\epsi}^{-}}, 
\end{align*}
where 
\begin{align*} 
&T_{\epsi}^{++}  = T_{\epsi}^{--} = Q - (1/\epsi)^2 I,  && C_{\epsi}^{+}  = D + (1/\epsi)\Theta,\\ 
&T_{\epsi}^{+-} = T_{\epsi}^{-+} = (1/\epsi)^2 I,  && C_{\epsi}^{-} = D - (1/\epsi)\Theta.
\end{align*} 
We assume that $\epsi$ is sufficiently small that the diagonal
elements of $C_{\epsi}^{+}$ are all positive, and those of
$C_{\epsi}^{-}$ are all negative, and we write $|C_{\epsi}^{-}|$  for
the matrix of absolute values of the elements of  $C_{\epsi}^{-}$. 

Let $\xi_\epsi^{+}(x) = \inf \{t < \infty: L_\epsi (t) > x\}$ and
$\xi_\epsi^{-}(x) = \inf \{t < \infty: L_\epsi (t) < x\}$ be the first
passage times to the level $x$, respectively from below and from
above, of the unbounded process $L_\epsi$. A key component of the stationary distribution of  $\{\widehat{L}_{\varepsilon}(t),
\beta_{\varepsilon}(t), \varphi_{\varepsilon}(t)\}$ is the matrix
$\Psi_\epsi$ of first passage probability from above, that is, 
\begin{align*} 
(\Psi_\epsi)_{ij} & =  \P[\xi_\epsi^{-}(x)< \infty, \beta_\epsi(\xi_\epsi^{-}(x)) = 2,
\varphi_\epsi(\xi_\epsi^{+}(x))=j\\
 & \qquad \quad | L_\epsi(0) = x, \beta_\epsi(0) = 1, \varphi_\epsi(0)=i]
\end{align*} 
for $i$ and $j$ in $\mathcal M$, and any level $x$.  Similarly, $\Psi_\epsi^*$ is the matrix of first passage probabilities from below, from $(x, 2, i)$ to $(x, 1,
j)$, for $i$ and $j$ in $\mathcal M$. 

The stationary distribution is given in the literature under various slightly different forms; here,  we use the
one from  Govorun \emph{et al.}~\cite[Theorem~2.1]{glr11b}:
\begin{align} 
   \label{eqn:smass}
\bs{F}_{\varepsilon}(0) & = \vligne{\bs{0} & \bs{\zeta}_{\varepsilon}}, \\
\label{eqn:sdensity} 
\bs{\pi}_{\varepsilon}(x) & = \bs{\zeta}_{\varepsilon}T^{-+}_{\varepsilon}e^{K_{\varepsilon} x}\vligne{(C^{+}_{\varepsilon})^{-1} & \Psi_{\varepsilon} |C_{\varepsilon}^{-}|^{-1}} \quad \mbox{ for } x > 0,
\end{align} 
where
\begin{equation}
   \label{eq:kepsi}
K_{\varepsilon} = (C_{\varepsilon}^{+})^{-1}T_{\varepsilon}^{++} + \Psi_{\varepsilon}|C_{\varepsilon}^{-}|^{-1}T^{-+}_{\varepsilon},
\end{equation}
and $\bs{\zeta}_{\varepsilon}$ is the unique solution of 
\begin{align} \label{eqn:zeta} 
& \bs{\zeta}_{\varepsilon} (T^{--}_{\varepsilon} + T^{-+}_{\varepsilon}\Psi_{\varepsilon}) = \bs{0}, \\
\label{eqn:zetanorm} 
& \bs{\zeta}_{\varepsilon} \bs{1} + \bs{\zeta}_{\varepsilon}T_{\varepsilon}^{-+}(-K_{\varepsilon})^{-1}\{(C^{+}_{\varepsilon})^{-1}\bs{1} + \Psi_{\varepsilon} |C_{\varepsilon}^{-}|^{-1} \bs{1}  \}= 1.
\end{align} 
Probabilistically, $\bs{\zeta}_{\varepsilon}$ is up to a multiplicative constant the stationary distribution of the
process censored at level~$0$, and $e^{K_\epsi x}$ is the matrix of expected number of crossings of level
$x$ in the various phases $(1,i)$, starting from level $0$, before the first return to level $0$.

In view of (\ref{eqn:sdensity}), we need to analyze $\Psi_\epsi$,
$K_\epsi$ and $\zeta_\epsi$ as $\epsi \rightarrow 0$, and it is obvious
from (\ref{eq:kepsi}) and (\ref{eqn:zeta}, \ref{eqn:zetanorm}) 
that we should focus on the matrix $\Psi_{\varepsilon}$ first.
The next lemma is the key to our analysis. One might expect
(\ref{eqn:Psieps}, \ref{eqn:Psistareps})  to have
a  simple proof but the one we have is lengthy and tedious.  We
place it in Appendix to preserve the flow of the paper.
Lemma~\ref{th:k} and Theorem~\ref{prop:limpi}  easily follow. 

\begin{lem} 
	\label{lem:expansions} 
For $\varepsilon \geq 0$,  
\begin{align} 
\Psi_{\varepsilon} = I + \varepsilon\Psi_1 + O(\varepsilon^2),  \label{eqn:Psieps} \\
\Psi^*_{\varepsilon} = I + \varepsilon\Psi^*_1 + O(\varepsilon^2), \label{eqn:Psistareps}
\end{align} 
where $\Theta^{-1}\Psi_1$ and $-\Theta^{-1}\Psi_1^*$ are both solutions to 
\begin{align} \label{eqn:theequation}
 X^2 + 2V^{-1}D X + 2V^{-1}Q= 0,
\end{align}
and are irreducible.
Furthermore,  the roots  $\theta_1, \theta_2, \ldots,$ $\theta_{2m}$ of the polynomial
\begin{align}
\gamma(z) = \det (z^2 I + 2z V^{-1}D +  2V^{-1}Q) \nonumber   
\end{align} 
associated to (\ref{eqn:theequation}),  numbered in increasing order
of their real parts, satisfy the inequalities
\begin{align}
   \label{e:splitting}
\mbox{Re}(\theta_1) \leq \cdots \leq \mbox{Re}(\theta_{m - 1}) < \theta_m  = 0  < \mbox{Re}(\theta_{m + 1}) \leq \cdots \leq \mbox{Re}(\theta_{2m}). 
\end{align} 
Finally, $\Theta^{-1}\Psi_1$ has one eigenvalue equal to zero and $m-1$ eigenvalues with strictly negative real parts, and it is the unique such solution; $-\Theta^{-1}\Psi_1^*$ has $m$ eigenvalues with strictly positive real parts, and is the unique such solution.
\qed  
\end{lem} 

\begin{rem} \emph{Let $\tau^{\pm}_x = \inf \{t < \infty: \pm Y(t) > x\}$ be the first passage times to the corresponding levels $x$ and $-x$ of the unbounded process $Y(t)$. Under Assumption~\ref{ass:possigma} that $\sigma_i > 0$ for all $i \in \mathcal{M}$, it is easy to confirm that $\Theta^{-1}\Psi_1$ and $\Theta^{-1}\Psi_1^*$ are the same as, respectively, the generators $\Lambda^{-}$ and $\Lambda^{+}$ of the time-changed processes $\kappa(\tau_x^-)$ and $\kappa(\tau_x^+)$ in Ivanovs~\cite{ivanovs-japs10}, and $\Theta^{-1}\Psi_1^*$ is the same as the matrix $U(\gamma)$ for $\gamma = 0$ in Breuer~\cite{breuer08}.}
\end{rem}

By Lemma~\ref{lem:expansions},
the matrices $\Psi_1$ and $\Psi_1^*$ are uniquely identified
through~\eqref{eqn:theequation}.  We now turn to the matrix $K_\epsi$,
and, to complete the picture, to the matrix $K_\epsi^*$ defined as
\begin{equation}
   \label{eq:kepsist}
K^*_{\varepsilon} = \Psi_{\varepsilon}^*(C_{\lambda}^{+})^{-1}T_{\lambda}^{+-} + |C_{\lambda}^{-1}|^{-1}T_{\lambda}^{--}. 
\end{equation}

\begin{lem}
   \label{th:k}
For $\epsi \geq 0$, 
\begin{align}
   \label{eq:kepsilon}
K_\epsi & = K_0 + O(\epsi), \\
   \label{eq:kstepsilon}
K^*_\epsi & = K^*_0 + O(\epsi),
\end{align}
 where $K_0 = \Psi_1 \Theta^{-1} + 2V^{-1}D$ and $K^*_0 =  \Psi^*_1  \Theta^{-1} - 2V^{-1}D$.  
The matrices  $-K_0$ and 
 $K^*_0$ are solutions of the equation
\begin{align} \label{eqn:forK0}
X^2 + 2 X V^{-1}D + 2 \Theta^{-1}Q \Theta^{-1} = 0,
\end{align} 
and are irreducible.
Furthermore, $K_0$ has $m$ eigenvalues with strictly negative real
parts, and it is the unique such solution; $ K^*_0$ has one
eigenvalue equal to zero and $m-1$ eigenvalues with strictly negative
real parts, and is the unique such solution.
\end{lem}

\begin{proof}
We write (\ref{eq:kepsi}) as 
\begin{align*}
\varepsilon K_{\varepsilon} & = -(\Theta + \varepsilon D)^{-1}(I - \varepsilon^2 Q) + \Psi_{\varepsilon}(\Theta - \varepsilon D)^{-1}\nonumber \\ 
& = -(\Theta^{-1} - \varepsilon V^{-1} D + O(\varepsilon^2))(I - \varepsilon^2 Q) \nonumber \\
& \;\;\;\; + (I + \varepsilon \Psi_1 + O(\varepsilon^2))(\Theta^{-1} + \varepsilon V^{-1} D + O(\varepsilon^2))  \nonumber \\
\intertext{by \eqref{eqn:expan1}, \eqref{eqn:expan2} and \eqref{eqn:Psieps},} 
& = \varepsilon (2 V^{-1} D + \Psi_1 \Theta^{-1}) + O(\varepsilon^2), 
\end{align*} 
which proves (\ref{eq:kepsilon});  equation (\ref{eq:kstepsilon})
follows in a similar manner. It is easy to verify that $K_0$ and  $-K_0^* $ both satisfy
(\ref{eqn:forK0}),
of which the associated polynomial is
\begin{align*}
\Xi(z) &= z^2 I + 2 z V^{-1} D + 2 \Theta^{-1} Q \Theta^{-1} \\
 & = \Theta^{-1} \Gamma(z) \Theta^{-1}  \qquad \mbox{with $\Gamma(z)$
   defined in Proposition \ref{t:factorization}} \\
 & = (zI +K_0) \Theta (z I - \Theta^{-1} \Psi_1)\Theta^{-1}
\end{align*}
by (\ref{e:whgamma}), after some simple manipulations.
This, together with Lemma~\ref{lem:expansions}, shows that the eigenvalues of $-K_0$ are the roots
$\theta_{m+1}$, \ldots, $\theta_{2m}$ of $\gamma(z)$ with strictly
positive real parts.  Finally, using (\ref{e:whgammab}) we write 
\begin{align*}
\Xi(z) = (zI - K_0^*) \Theta (zI + \Theta^{-1} \Psi_1^*) \Theta^{-1},
\end{align*} 
and conclude that the eigenvalues of $K_0^*$ are the roots $\theta_1$
to $\theta_m$.  Finally, as $\Psi_1$ and $\Psi_1^*$ are irreducible,
and $\Theta$, $V$, and $D$ are diagonal matrices, we conclude that
$K_0$ and $K_0^*$ are irreducible, and this completes the proof.
\end{proof}

The next theorem states that the limit, as $\lambda \rightarrow
\infty$, of the stationary distributions of the approximating fluid
processes is indeed the stationary distribution of the limiting
process $\{Y(t), \kappa(t)\}$. We prove this result by showing that
the limiting distribution 
(\ref{eqn:theend})
coincides with the stationary distribution of $\{Y(t), \kappa(t)\}$ as obtained by Asmussen~\cite[Theorem~2.1 and Corollary~4.1]{asmussen95}.

\begin{thm} 
	\label{prop:limpi} 
	The limiting distribution of
  $\{\widehat L_\lambda(t), \widehat\varphi_\lambda(t)\}$ converges,
  as $\lambda$ goes to infinity, to the stationary distribution of
  $\{Y(t),\kappa(t)\}$, and is given by
\begin{align}  \label{eqn:theend}
\lim_{\varepsilon \rightarrow 0} {\bs{\pi}}_{\varepsilon}(x) (\bs 1 \otimes I) & = 2 \bs{\zeta}_1e^{K_0 x} \Theta^{-1}, \\
   \label{eqn:theendb}
\lim_{\varepsilon \rightarrow 0} \bs{F}_{\varepsilon}(0)(\bs{1} \otimes I) & = \bs{0},
\end{align}
where 
\begin{align}
& \bs{\zeta}_1\Psi_1  = \bs{0}, \label{eqn:xeta1c1} \\
& 2\bs{\zeta}_1(-K_0)^{-1}\Theta^{-1} \bs{1}  = 1. \label{eqn:xeta1c2}
\end{align} 
\end{thm}
\begin{proof} The solution of (\ref{eqn:zeta}) is of the form $\bs{\zeta}_{\varepsilon} = \bs{\zeta}_0 + \varepsilon \bs{\zeta}_1 + o(\varepsilon)$ (\cite[Theorem~5.4]{kato66}) and \eqref{eqn:zetanorm} becomes
\begin{align}
& \{\bs{\zeta}_0 + \varepsilon\bs{\zeta}_1 + o(\varepsilon)\}\bs{1} \nonumber \\
& + \{\bs{\zeta}_0 + \varepsilon\bs{\zeta}_1 + o(\varepsilon)\}(1/\varepsilon)(-K_{\varepsilon})^{-1}\{(\varepsilon D + \Theta)^{-1} + \Psi_{\varepsilon}(\Theta - \varepsilon D)^{-1}\}\bs{1} = 1.  \label{eq:normalzeta}
\end{align} 
We equate the coefficients of $1/\epsi$ on both sides of
(\ref{eq:normalzeta}), using (\ref{eqn:Psieps}, \ref{eq:kepsilon}) and
we find that 
\begin{align*} 
-2 \bs\zeta_0 K_0^{-1} \Theta^{-1} \bs 1= 0. 
\end{align*} 
Equation \eqref{eqn:smass} implies that $\bs{\zeta}_{\varepsilon} \geq
\bs{0}$ and, by continuity, $\bs{\zeta}_{0} \geq \bs{0}$. Furthermore,
$\Theta^{-1}$ is a diagonal matrix with strictly positive diagonal.
Finally, $K_0$ is irreducible and has eigenvalues with strictly
negative real part by Lemma~\ref{th:k}, so that
$\int_0^{\infty}e^{K_0 u} \ud u$ converges to $-K_0^{-1}$ and is
strictly positive.
This implies that $\bs{\zeta}_0 = \bs{0}$, which
proves~\eqref{eqn:theendb}. Using $\bs{\zeta}_0 = \bs{0}$, and equating
coefficients of $\varepsilon$ on both sides of \eqref{eqn:zeta} gives
\eqref{eqn:xeta1c1}, while  equating the coefficients of $\epsi^0$ on
both sides of (\ref{eq:normalzeta}) leads to (\ref{eqn:xeta1c2}).

Gathering everything together, we obtain from (\ref{eqn:sdensity})
\begin{align} 
\bs{\pi}_{\varepsilon}(x) & = (1/\varepsilon)\{\varepsilon \bs{\zeta}_1 + o(\varepsilon)\}e^{K_{\varepsilon}x}\vligne{(\varepsilon D + \Theta)^{-1} & \Psi_{\varepsilon}(\Theta - \varepsilon D)^{-1}},
\end{align} 
from which (\ref{eqn:theend}) follows.  Equation (\ref{eqn:theendb}) is a direct consequence of (\ref{eqn:smass}).

To verify that the limiting distribution in Theorem~\ref{prop:limpi} is the stationary density vector~$\bs{g}(x)$ of the Markov-modulated Brownian motion, we use  Asmussen~\cite{asmussen95}.  By \cite[Theorem~2.1 and Corollary~4.1]{asmussen95}, 
\begin{align} 
\bs{g}(x) & = [e^{\Lambda x}(-\Lambda \bs{1})]^{\mbox{\tiny T}}\Delta_{\bs{\alpha}}, \label{eqn:Asmusseng1}
\end{align}
where $\bs{\alpha}$ is the stationary distribution vector of $Q$, $\Delta_{\bs{\alpha}} = \diag(\bs{\alpha})$, and $\Lambda$ is a defective generator matrix satisfying 
\begin{align}
(1/2)V\Lambda^2 - D\Lambda + \Delta_{1/\bs{\alpha}} Q^{\mbox{\tiny T}} \Delta_{\bs{\alpha}} = 0. \label{eqn:Lambda}
\end{align} 
Define $Z = \Delta_{1/{\bs{\alpha}}} \Lambda^{\mbox{\tiny T}} \Delta_{\bs{\alpha}}$ and rewrite \eqref{eqn:Asmusseng1} as 
\begin{align}
\bs{g}(x) & = -\bs{1}^{\mbox{\tiny T}} \Delta_{\bs{\alpha}}Z\Delta_{1/\bs{\alpha}} e^{\Delta_{\bs{\alpha}}Z\Delta_{1/\bs{\alpha}}x}\Delta_{\bs{\alpha}}  = -\bs{\alpha} Z e^{Zx}. \label{eqn:Asmusseng2}
\end{align} 
By (\ref{eqn:Lambda}), we find that 
\[
(1/2) Z^2 V - Z D + Q =0
\]
and that $\Theta^{-1} Z \Theta$ is a solution of \eqref{eqn:forK0}.
It is similar to $Z$ and so to $\Lambda^{\mbox{\tiny T}} $, therefore,
the  eigenvalues of $\Theta^{-1} Z \Theta$ 
all have strictly negative real parts and we have
\begin{align} \label{eqn:ZK} 
K_0 = \Theta^{-1} Z \Theta.
\end{align} 
Substituting \eqref{eqn:ZK} into \eqref{eqn:theend} gives $\lim_{\varepsilon \rightarrow 0} \bs{\pi}_{\varepsilon}(\bs{1} \otimes I) = 2\bs{\zeta}_1\Theta^{-1}e^{Zx}.$ Finally, it is straightforward to verify that $\bs{\zeta}_1 = -\bs{\alpha}(\Theta^{-1}D + (1/2)\Theta\Psi_1 \Theta^{-1})$, and consequently 
\begin{align}
\lim_{\varepsilon \rightarrow 0} \bs{\pi}_{\varepsilon}(\bs{1} \otimes I) = \bs{g}(x).
\end{align} 
\end{proof}

\begin{rem} \em An alternative way to show that the stationary
  distribution of the approximating fluid process
  $\{\widehat{L}_{\lambda}(t), \widehat{\varphi}_{\lambda}(t)\}$
  converges, as $\lambda \rightarrow \infty$, to the stationary
  distribution of the limiting Markov-modulated Brownian motion
  $\{Y(t), \kappa(t)\}$ is via the maximum representation of the
  relevant processes.  Asmussen~\cite{asmussen95} derives the
  stationary distribution, both for fluid queues and for the MMBM in
  this manner, linking these to the distribution of the maximum of the
  time-reversed process.  Following the arguments in Enikeeva {\em et
    al.} \cite{ekr01} and in Stenflo \cite{stenf01}, one might show
  that there is continuity of the maximum distributions of the
  backward processes, as $\lambda \rightarrow \infty$, and
  consequently obtain the continuity of stationary distributions.

  This would lead to a time reversal-based proof of convergence.  As
  stated in the introduction, we aim at following the forward-time
  approach and so obtain a different representation of the stationary
  distribution.  In addition, we obtain limiting properties for key
  matrices, and these results will be proved useful in future work.
\end{rem}

\section{Computational procedure} 
	\label{sec:comproc}
        
Theorem~\ref{prop:limpi} indicates that the matrix $\Psi_1$ is the
central ingredient in evaluating the stationary distribution of the
Markov-modulated Brownian motion $\{Y(t), \varphi(t)\}$. We describe
here how to use the splitting property (\ref{e:splitting}) in
numerically solving for $\Psi_1$ and $\Psi_1^*$.   

Bini and Gemignani~\cite{bg05} consider
 quadratic matrix equations $C + AX + BX^2 = 0$  where the roots of
the associated polynomial $\det(C + zA + z^2B)$ are split by a circle
in $\C$, half being inside the disk and half outside.  The problem
in~\cite{bg05} 
is to find the minimal solution, that is, the solution matrix with all
eigenvalues inside the disk.

In our case, the roots are split between the negative and the positive
half-planes and we need to apply some transformation, such as the one
described in Bini \emph{et al.}~\cite{blm01b} and based on the inverse
M\"{o}bius mapping~\cite[Chapter~2.1]{apostol97}
\begin{align}
w(z) = \frac{z - 1}{z + 1}.  
\end{align} 
This inverse mapping applies the open unit disk $|z|<1$ onto the
negative half-plane $\C_\d$, the unit circle $|z|=1$ minus the point
$z=-1$ onto the imaginary axis $\C_0$, the outside $|z|>1$ of the
closed unit disk onto the positive half-plane $\C_\u$, and the
imaginary axis $\C_0$ onto the unit circle $|w|=1$ minus the point
$w=1$.

Now, define $W(Z) = (Z - I)(Z + I)^{-1}$.  Instead of solving
$P(X) = X^2 + 2V^{-1}D X + 2V^{-1}Q = 0$ for $\Theta^{-1}\Psi_1$, we
solve $H(Z) = 0$, where
\begin{align*} 
H(Z) & = P(W(Z))(I + Z)^2  \\
        & = P((Z - I)(Z + I)^{-1})(I + Z)^2   \\
        & = (I + 2V^{-1}D + 2V^{-1}Q)Z^2 - 2(I - 2V^{-1}Q)Z + I \\
         & \quad -         2V^{-1}D + 2V^{-1}Q.   
\end{align*} 
The roots of $\det(H(z))$ are given by $\omega_i = w^{-1}(\tau_i) = (1
+ \tau_i)/(1 - \tau_i)$ for $i = 1, \ldots, 2m$, and satisfy the
splitting property
\begin{align} \label{eqn:omegai}
0 \leq  |\omega_1| , \ldots , |\omega_{m - 1}| < \omega_m = 1 < |\omega_{m + 1}| , \ldots , |\omega_{2m}|.
\end{align} 
We note that Bini and Geminiani~\cite{bg05} requires $|\omega_i| >0$ for all $i$, but as seen in \cite[Section 8.3]{blm05}, the weak inequality suffices.

Define $Z_0$ as the solution of $H(Z)=0$ such that $\sp(Z_0) \leq 1$.
The matrix $W(Z_0)=(Z_0-I)(Z_0+I)^{-1}$ is a solution of
$P(X)=0$ with all eigenvalues in $\{\re(z) \leq 0\}$, and so
$\Theta^{-1}\Psi_1 = (Z_0-I)(Z_0+I)^{-1}$.  

Several iterative algorithms to compute $Z_0$ are discussed in
\cite{bg05}.  Some have superlinear convergence, such as Cyclic
reduction~\cite{bgm01}, Logarithmic reduction~\cite[Chapter 8]{lr99},
subspace iteration \cite{as97} or Graeffe iteration \cite{bgm01}.
This means that the approximation error at the $i$th iteration
is $O(\sigma^{2^i})$ with $\sigma = 1/ |\omega_{m+1}| < 1$.
These algorithms are globally convergent but they are not
self-correcting.  Since the coefficient 
matrices in $P(X)$ are of mixed signs, one might prefer the algorithm
developed in \cite{bg05}: it is self-correcting and the approximation
error is $O(\sigma^{i 2^k})$ for arbitrary $k$, which makes it
arbitrarily fast.

As for $\Psi_1^*$, if we define $Z_1$ as the root of $H(Z)$ such that
all of its eigenvalues are outside the closed unit disk, then $W(Z_1)$
has all its eigenvalues in the half-plane $\C_\u$, and
$\Theta^{-1}\Psist_1 = -W(Z_1)$. The algorithms in \cite{bg05},
however, do not seem to be well adapted to the computation of $Z_1$,
and we suggest to use a different transformation, in order to bring
the eigenvalues of $-\Theta^{-1}\Psist_1$ inside the unit disk. This
transformation is based on M\"{o}bius'
mapping~\cite[Chapter~2.1]{apostol97}
\begin{align} 
z(w) = \frac{1 + w}{1 - w}.
\end{align} 
This mapping applies the open unit disk $|w|<1$ onto the positive half-plane $\C_\u$, the unit circle $|w|=1$ minus the point $w=1$ onto the imaginary axis $\C_0 = \{z: \re(z)=0\}$ and the outside $|w|>1$ of the closed unit disk onto the negative half-plane $\C_\d$, finally, it applies the imaginary axis $\C_0$ onto the unit circle $|z|=1$ minus the point $z=-1$.
 
Now, define $Z(W) = (I + W)(I - W)^{-1}$. Instead of solving $H(Z) =
0$ we solve $Q(W) = 0$ for the matrix solution $W_1$ with eigenvalues
inside the unit disk, where
%
%
\begin{align*}
Q(W) & = P(Z(W))(I - W)^2, \\
         & = (I - 2V^{-1}D + 2V^{-1}Q)W^2  + 2(I - 2V^{-1}Q)W + I + 2V^{-1}D + 2V^{-1}Q,
\end{align*} 
and so $\Theta^{-1} \Psi_1^* = -(I+W_1)(I-W_1)^{-1}$.
	

\appendix

\section{Proof of Lemma \ref{lem:expansions}}
The proof goes in four main steps.  The matrix $\Psi_\epsi$ is the
stochastic (or substochastic) solution of the Riccati equation 
\[
(C_{\varepsilon}^{+})^{-1}  T^{+-}_{\varepsilon} 
+ (C_{\varepsilon}^{+})^{-1} T^{++}_{\varepsilon} \Psi_\epsi
+ \Psi_\epsi |C_{\varepsilon}^{-}|^{-1} T^{--}_{\varepsilon} 
+ \Psi_\epsi |C_{\varepsilon}^{-}|^{-1}
T^{-+}_{\varepsilon}\Psi_{\varepsilon}
= 0
\]
(Rogers~\cite{roger94}), equation that we write as
\begin{align}
& (1/\varepsilon)(\varepsilon D + \Theta)^{-1} + (1/\varepsilon)(\varepsilon D + \Theta)^{-1}(\varepsilon^2 Q - I) \Psi_{\varepsilon} \nonumber \\
& \hspace*{1cm} + ({1}/{\varepsilon}) \Psi_{\varepsilon} |\varepsilon D - \Theta|^{-1} (\varepsilon^2 Q - I) + ({1}/{\varepsilon})\Psi_{\varepsilon} |\varepsilon D - \Theta|^{-1} \Psi_{\varepsilon} = 0 \label{eqn:coffee}
\end{align} 
Thus, $\Psi_\epsi$ is a solution of $\FF_\epsi(X) =0$, where
\begin{align*}
\FF_\epsi(X) = & (\varepsilon D + \Theta)^{-1} + (\varepsilon D + \Theta)^{-1}(\varepsilon^2 Q - I) X  \\
& + X |\varepsilon D - \Theta|^{-1} (\varepsilon^2 Q - I) + X |\varepsilon D - \Theta|^{-1} X.
\end{align*} 
For $\epsi = 0$, we see that $\FF_0(I)=0$.  It  is tempting to invoke
the Implicit Function Theorem and claim that $\Psi_\epsi$ is an
analytic function of $\epsi$ in a neighborhood of $\epsi=0$.
Unfortunately, the operator $\partial/\partial X \FF_\epsi(X)$ is
singular at the point $(\epsi=0, X=I)$, the Implicit Function Theorem
does not apply, and we follow a longer,
more tortuous path. 

\begin{prop} \label{prop:updown} For $\varepsilon \geq 0$, 
\begin{align}
\Psi_{\varepsilon} & = I + \Phi_{\varepsilon}  \quad \mbox{ where } \lim_{\varepsilon \rightarrow 0} \Phi_{\varepsilon} = 0,  \label{eqn:up} \\
\Psi_{\varepsilon}^* & = I + \Phi_{\varepsilon}^* \quad \mbox{ where } \lim_{\varepsilon \rightarrow 0} \Phi_{\varepsilon}^* = 0. \label{eqn:down} 
\end{align} 
\end{prop} 

\begin{proof}
We note that 
\begin{align}
(\varepsilon D + \Theta)^{-1} 
                & = \{I - \varepsilon \Theta^{-1}(-D)\}^{-1} \Theta^{-1} \nonumber \\
                & = \{I - \varepsilon \Theta^{-1} D + \varepsilon^2(\Theta^{-1}D)^2 + O(\varepsilon^3)\}\Theta^{-1} \nonumber \\
                & = \Theta^{-1} - \varepsilon V^{-1} D + \varepsilon^2\Theta^{-1}V^{-1} D^2 + O(\varepsilon^3), \label{eqn:expan1}
\intertext{and similarly} 
|\varepsilon D - \Theta|^{-1} & = \Theta^{-1} + \varepsilon V^{-1} D + \varepsilon^2 \Theta^{-1} V^{-1} D^2 + O(\varepsilon^3), \label{eqn:expan2} 
\end{align}
for $\varepsilon$ sufficiently small.  
Thus, \eqref{eqn:coffee} implies that
\begin{align}
& ({1}/{\varepsilon})\{\Theta^{-1} - \varepsilon V^{-1}D + O(\varepsilon^2)\}(I  + \varepsilon^2Q \Psi_{\varepsilon} - \Psi_{\varepsilon}) \nonumber \\
& \quad\quad + ({1}/{\varepsilon})\Psi_{\varepsilon}\{\Theta^{-1} + \varepsilon V^{-1}D + O(\varepsilon^2)\}(\varepsilon^2Q - I + \Psi_{\varepsilon}) = 0, \label{eqn:coffee2} 
\end{align} 
which can be reorganized as $\mathcal{G}(\varepsilon) + \mathcal{H}(\varepsilon) = 0$, where 
\begin{align}
\label{eqn:G}
\mathcal{G}(\varepsilon) & =  ({1}/{\varepsilon})(\Theta^{-1} -
\Theta^{-1} \Psi_{\varepsilon} - \Psi_{\varepsilon}\Theta^{-1} +
\Psi_{\varepsilon}\Theta^{-1} \Psi_{\varepsilon}) \\
\intertext{and}
\mathcal{H}(\epsi) & = ({1}/{\varepsilon})\{\varepsilon^2 \Theta^{-1} Q \Psi_{\varepsilon} + (-\varepsilon V^{-1}D + O(\varepsilon^2))(I + \varepsilon^2 Q \Psi_{\varepsilon} - \Psi_{\varepsilon}) \nonumber \\
&\quad  + \varepsilon^2 \Psi_{\varepsilon}\Theta^{-1} Q +  \Psi_{\varepsilon}(\varepsilon V^{-1} D + O(\varepsilon^2))(\varepsilon^2Q - I + \Psi_{\varepsilon})\}. \label{eqn:H}
\end{align}
The matrix
$\mathcal{H}(\epsi)$ is bounded and therefore $\GG(\epsi)$ too remains bounded as $\varepsilon \rightarrow 0$. 

Now, we observe that $\Psi_{\varepsilon}$ belongs to the compact set
$\{M: M \geq 0, M\bs{1}\leq \bs{1}\}$ of (sub)stochastic matrices;
therefore, for every sequence $\{\Psi_{\varepsilon}\}_{\varepsilon
  \rightarrow 0}$ there exist subsequences that converge. Let
$\bar{\Psi}$ be the limit of one such convergent subsequence, and
$\{\varepsilon_i\}_{i = 1, 2, \ldots}$ be a sequence such that
$\varepsilon_i \rightarrow 0$ and $\Psi_{\varepsilon_i} \rightarrow
\bar{\Psi}$ as $i \rightarrow \infty$. Since
$\mathcal{G}(\varepsilon_i)$ remains bounded as $i \rightarrow
\infty$, necessarily
\begin{align*}
\lim_{i \rightarrow \infty} (\Theta^{-1} - \Theta^{-1}\Psi_{\varepsilon_i} - \Psi_{\varepsilon_i}\Theta^{-1}  + \Psi_{\varepsilon_i} \Theta^{-1} \Psi_{\varepsilon_i}) & = \lim_{i \rightarrow \infty} (I - \Psi_{\varepsilon_i})\Theta^{-1}(I - \Psi_{\varepsilon_i})\\
& =  (I - \bar{\Psi})\Theta^{-1}(I - \bar{\Psi}) \\
& = 0,
\end{align*} 
and thus $\bar{\Psi} = I$. This follows from the facts that $\Theta^{-1}(I - \bar{\Psi})$ is a nilpotent matrix, that the trace of every nilpotent matrix is zero, and that $\bar{\Psi}$ is a (sub)stochastic matrix while $\Theta^{-1}$ is a strictly positive diagonal matrix. 

All convergent subsequences having the same limit, the conclusion is that $\Psi_{\varepsilon}$ converges to $I$ as $\varepsilon \rightarrow 0$, and \eqref{eqn:up} follows.  The proof of \eqref{eqn:down} is
by analogous arguments.
\end{proof} 

\begin{prop}
   \label{t:generator}
For $\epsi >0$, the matrices $\Phi_\epsi$ and $\Phi_\epsi^*$ are
irreducible with non-negative off-diagonal elements, strictly negative diagonal elements, $\Phi_\epsi \vone
\leq \vzero$ and $\Phi_\epsi^* \vone \leq \vzero$.
In addition, under
Assumption~\ref{a:drift}, $\Phi_\epsi \vone = \vzero$ and
$\Phi_\epsi^* \vone < \vzero$.

In short, $\Phi_\epsi$ is an irreducible generator and $\Phi_\epsi^*$
is an irreducible subgenerator.
\end{prop}

\begin{proof}
  As we assume that the fluid queue is irreducible,
  $\Psi_{\varepsilon}$ is an irreducible (sub)stochastic 
  matrix for all~$\varepsilon \geq 0$. Thus, we conclude from
  (\ref{eqn:up}) that $\Phi_\epsi$   is irreducible and that its
  off-diagonal elements  are
  nonnegative.  Furthermore, since $\Psi_{\varepsilon}\bs{1} 
  \leq \bs{1}$, this implies that $\Phi_\epsi \bs{1} \leq \bs{0}$, so
  that its diagonal elements are strictly negative and $\Phi_\epsi$ is
  a generator.  The same argument holds for $\Phi_\epsi^*$.

Under Assumption~\ref{a:drift}, the matrix $\Psi_\epsi$ is stochastic
and the matrix $\Psi_\epsi^*$ is strictly substochastic, and the last
claim follows.
\end{proof}

\begin{prop} \label{prop:phi}
  The matrices $(1/\varepsilon) \Phi_{\varepsilon}$ and
  $(1/\varepsilon)\Phi_{\varepsilon}^*$ are bounded.  Denoting by
  $\bar{\Psi}_1$ and $\bar{\Psi}_1^*$ the limits of any converging
  subsequences of $(1/\varepsilon) \Phi_{\varepsilon}$ and
  $(1/\varepsilon)\Phi_{\varepsilon}^*$ respectively, both
  $\bar{\Psi}_1$ and $-\bar{\Psi}_1^*$ are solutions of the equation
\begin{equation}
   \label{eqn:sun}
(\Theta^{-1} Y)^2 + 2 V^{-1} D \Theta^{-1} Y +  2 V^{-1} Q = 0,
\end{equation}
and are irreducible.
\end{prop} 
\begin{proof} 
Substituting \eqref{eqn:up} into (\ref{eqn:G}, \ref{eqn:H})  gives us 
\begin{align*}
\mathcal{G}(\varepsilon) 
& =  ({1}/{\varepsilon}) \{\Phi_{\varepsilon}\Theta^{-1}\Phi_{\varepsilon}\}
\\
 \mathcal{H}(\varepsilon) & = \varepsilon \Theta^{-1} Q (I + \Phi_{\varepsilon}) + [-V^{-1}D + O(\varepsilon)](\varepsilon^2 Q + \varepsilon^2 Q \Phi_{\varepsilon}- \Phi_{\varepsilon})  \\
 & \;\;\;\; + \varepsilon (I + \Phi_{\varepsilon}) \Theta^{-1} Q + (I + \Phi_{\varepsilon})[V^{-1} D + O(\varepsilon)](\varepsilon^2Q + \Phi_{\varepsilon}).
\end{align*} 
Clearly, $\lim_{\epsi \rightarrow 0} \HH(\epsi) = 0$ and this implies that 
$\lim_{\epsi \rightarrow 0}({1}/{\varepsilon}) \{\Phi_{\varepsilon}\Theta^{-1}\Phi_{\varepsilon}\}=0$ since $\mathcal{G}(\varepsilon) + \mathcal{H}(\varepsilon) = 0$.  Divide both sides of that equation by $\epsi$ and obtain
\begin{equation} \label{eqn:laststep}
({1}/{\varepsilon)^2} \Phi_{\varepsilon} \Theta^{-1}\Phi_{\varepsilon} + 2\Theta^{-1}Q + \mathcal{R}(\varepsilon) = 0
\end{equation} 
where
\begin{align} 
\mathcal{R}(\varepsilon) & =  \Theta^{-1}Q\Phi_{\varepsilon}  + [-V^{-1}D + O(\varepsilon)](\varepsilon Q + \varepsilon Q \Phi_{\varepsilon} - ({1}/{\varepsilon})\Phi_{\varepsilon}) \nonumber \\
& \;\;\;\; + \Phi_{\varepsilon} \Theta^{-1}Q + (I + \Phi_{\varepsilon})[V^{-1}D + O(\varepsilon)]\{\varepsilon Q + ({1}/{\varepsilon})\Phi_{\varepsilon}\}.
\label{eqn:R}
\end{align} 
Now, there are three possible cases. 

\paragraph{Case 1:} $(1/\varepsilon) \Phi_{\varepsilon} \rightarrow 0$ as $\varepsilon \rightarrow 0$. Then, $\mathcal{R}(\varepsilon) \rightarrow 0$ and taking the limit of both sides of \eqref{eqn:laststep} as $\varepsilon \rightarrow 0$ leads to $2V^{-1}Q = 0$, which is not true. 

\paragraph{Case 2:} $(1/\varepsilon) \Phi_{\varepsilon}$ is unbounded
in any neighborhood of $\varepsilon = 0$. Then, there exists a
sequence $\{\varepsilon_k\}_{k = 1, 2 \ldots}$ such that
$\varepsilon_k \rightarrow 0$ and
$\max_{ij}|\Phi_{\varepsilon_k}|_{ij} / \varepsilon_k\rightarrow
\infty$. 
In this case, we may write $(1/\varepsilon)\Phi_{\varepsilon} =
u_{\varepsilon} B_{\varepsilon}$, where $u_{\varepsilon}$ is a scalar
function such that $\lim_{k \rightarrow \infty} u_{\varepsilon_k} =
\infty$ while $B_{\varepsilon_k}$ remains bounded and does not
converge to zero: 
since $\Phi_\epsi$
is an irreducible generator, its maximum element is on the diagonal
and we 
take $u_{\varepsilon} =
\max_{j}|\Phi_{\varepsilon}|_{jj} / \varepsilon$,
then $B_\epsi$ is an irreducible generator with at least one diagonal
element equal to $-1$, and with $|B_{ij}| \leq 1$ for all $i$ and $j$. 

Next, for $\epsi$
in the sequence $\{\varepsilon_k\}$, we replace $(1/\epsi) \Phi_\epsi$
in~\eqref{eqn:laststep} by $u_\epsi B_\epsi$ and divide both sides of
the equation by $u_\epsi^2$ to obtain
\begin{eqnarray}
\lefteqn{ B_{\varepsilon} \Theta^{-1} B_{\varepsilon} + (2/u_\epsi^2) \Theta^{-1} Q + (\varepsilon /u_{\varepsilon}) \Theta^{-1} Q B_{\varepsilon}}  \nonumber \\
& & + [-V^{-1}D + O(\varepsilon)](\varepsilon / u_\epsi^2 Q + \varepsilon^2 /u_{\varepsilon} Q B_{\varepsilon} - (1/u_{\varepsilon})B_{\varepsilon}) \nonumber \\
& & +  ((1/u_\epsi) I + \varepsilon B_{\varepsilon})[V^{-1}D + O(\varepsilon)]((\varepsilon/u_\epsi) Q + B_{\varepsilon}) = 0, \label{eqn:reallaststep}
\end{eqnarray} 
This implies that 
\begin{equation}
   \label{e:BB}
\lim_{k \rightarrow \infty} B_{\varepsilon_k} \Theta^{-1} B_{\varepsilon_k} = 0,
\end{equation}
Now, take any converging subsequence of $B_\epsi$ and denote its limit as
$B$.  By construction, the trace of $\Theta^{-1} B_\epsi$ is at most
equal to $\min_j (-\sigma_j^{-1}) < 0$, independently of $\epsi$.
Thus, the trace of $\Theta^{-1} B$ is strictly negative, the matrix
$\Theta^{-1} B$ is not nilpotent, and $\Theta^{-1} B\Theta^{-1} B
\not= 0$, which contradicts (\ref{e:BB}).

\paragraph{Case 3:} $(1/\varepsilon)\Phi_{\varepsilon}$ is bounded and
does not converge to  $0$. Then, from \eqref{eqn:R}
\[
\mathcal{R}(\varepsilon)  =  ({1}/{\varepsilon}) 2V^{-1}D \Phi_{\varepsilon} + \mathcal{R}^*(\varepsilon) 
\]
where $\mathcal{R}^*(\varepsilon)$ goes to 0 as $\epsi$ goes to zero.
This allows us to rewrite \eqref{eqn:laststep} as 
\[
({1}/{\varepsilon^2})\Phi_{\varepsilon} \Theta^{-1}\Phi_{\varepsilon} + 2\Theta^{-1}Q + ({1}/{\varepsilon}) 2V^{-1}D \Phi_{\varepsilon} + \mathcal{R}^*(\varepsilon)  = 0
\]
and to conclude that    
\begin{equation}
   \label{eq:psiepse}
\lim_{\epsi \rightarrow 0} \{
({1}/{\varepsilon^2})\Phi_{\varepsilon} \Theta^{-1}\Phi_{\varepsilon} + 2\Theta^{-1}Q + ({1}/{\varepsilon}) 2V^{-1}D \Phi_{\varepsilon} + \mathcal{R}^*(\varepsilon) \} = 0.
\end{equation}
Since $(1/\epsi) \Phi_\epsi$ is bounded, there exist subsequences
$\{\epsi_k\}_{k=1, 2, \ldots}$ such that $\epsi_k \rightarrow 0$ and
such that $(1/\epsi_k) \Phi_{\epsi_k} \rightarrow \bar{\Psi}_1$.  We
take in (\ref{eq:psiepse}) the limit along such a subsequence and
conclude that $\bar{\Psi}_1$ is a solution of (\ref{eqn:sun}).  The
same approach is followed for $\Phi^*_\epsi$.

Now, assume that $\bar\Psi_1$ is reducible.  We may write
\[
\Theta^{-1} \bar\Psi_1 = \vligne{M_A & 0 \\ M_{AB}  & M_B},
\]
possibly after a permutation of rows and columns, where $M_A$ and
$M_B$ are square matrices.   As $\Theta^{-1} \bar\Psi_1$ is a solution
of (\ref{eqn:sun}), we are led to conclude that
\[
Q = \vligne{Q_A & 0 \\ Q_{AB}  & Q_B}
\]
which contradicts our assumption that $Q$ is irreducible.  Thus,
$\bar\Psi_1$ is irreducible, and so is $\bar\Psi_1^*$ by the same
argument.
\end{proof} 

\begin{prop}
   \label{t:factorization}
   Consider the matrix equation 
\begin{equation}
   \label{e:factor}
V X^2 +2DX + 2Q =0
\end{equation}
and its associated matrix polynomial $\Gamma(z) = V z^2 + 2 D z +2Q$.

Under Assumption \ref{a:drift}, $\det \Gamma(z)$ has one root equal to zero,
$m-1$ roots with strictly negative real parts, and $m$ roots with
strictly positive real parts.
\end{prop}

\begin{proof}
Take
  $\{\epsi_k\}_{k=1, 2, \ldots}$ to be a subsequence such that
  $\epsi_k \rightarrow 0$ and  $(1/\epsi_k) \Phi_{\epsi_k} \rightarrow
  \bar{\Psi}_1$.  By Proposition \ref{prop:phi}, $\Theta^{-1}
  \bar{\Psi}_1$ is an irreducible generator and a solution of
  \eqref{e:factor},
and we write 
\begin{align}
  \nonumber
\Gamma(z) & = z^2 V + 2 D z +2Q - V((\Theta^{-1} \bar{\Psi}_1)^2 + 2V^{-1} D \Theta^{-1} \bar{\Psi}_1 +  2 V^{-1} Q)  \\
   \label{e:whgamma}
 & = (Vz + V \Theta^{-1} \bar{\Psi}_1 + 2D)(zI - \Theta^{-1} \bar{\Psi}_1).
\end{align}
We conclude that all eigenvalues of $\Theta^{-1} \bar{\Psi}_1$ are
roots of $\det \Gamma(z)$. 
As we assume that the fluid queue is positive recurrent, $\Phi_\epsi
\vone = \vzero$ for all $\epsi$ by Proposition \ref{t:generator} and
so $\Theta^{-1} \bar{\Psi}_1 \vone = \vzero$.  Hence, 
$\det \Gamma(z)$ has at least one root equal to zero, and at least $m - 1$
roots with strictly negative real parts.

In a similar manner, we take a subsequence $\{\varepsilon^*_k\}_{k =
  1, 2, \ldots}$ such that $\varepsilon^*_k \rightarrow 0$ and
$(1/\varepsilon^*_k)\Phi^*_{\varepsilon^*_k} \rightarrow
\bar{\Psi}^*_1$, and we show that
\begin{equation}
   \label{e:whgammab}
\Gamma(z) = (Vz - V \Theta^{-1} \bar{\Psi}^*_1 + 2D)(zI + \Theta^{-1} \bar{\Psi}^*_1).
\end{equation}
Therefore, all eigenvalues of $-\Theta^{-1}
\bar{\Psi}^*_1$ are roots of $\det \Gamma(z)$ and, since $\bar{\Psi}^*_1 \vone
< \vzero$, this shows that $\det \Gamma(z)$ has at least $m$ roots with
strictly positive real part.

The polynomial $\det \Gamma(z)$ has at most $2m$ roots, which concludes the proof.
\end{proof}

\vspace{1\baselineskip}

\noindent
Now we are ready to conclude.   
  By the properties of the roots of $\det \Gamma(z)$ given in
  Proposition~\ref{t:factorization}, \eqref{eqn:theequation} has one
  unique solution suitable for the role of $\bar{\Psi}_1$ and another
  unique solution suitable for the role of
  $\bar{\Psi}^*_1$. Consequently, all convergent subsequences give the
  same limit $\Psi_1$ for $(1/\varepsilon_k)\Phi_{\varepsilon_k}$, and
  $\Psi_1^*$ for $(1/\varepsilon_k)\Phi^*_{\varepsilon_k}$.
\qed 
	

\bibliographystyle{abbrv}
\bibliography{LNR}

\end{document}